\documentclass[12pt,leqno]{amsart}
\usepackage{amsmath,amsfonts,amssymb,amsthm,amscd,mathrsfs}
\usepackage{graphicx}
\graphicspath{ }
\usepackage{wrapfig}
\usepackage{subcaption}
\usepackage{calligra}

\numberwithin{figure}{section}

\theoremstyle{plain}
\newtheorem{theorem}{Theorem}[section]
\newtheorem{lemma}[theorem]{Lemma}

\newtheorem{cor}{Corollary}[theorem]
\theoremstyle{definition}

\theoremstyle{remark}

\usepackage{mathtools}
\title[$(m,\rho)$-quasi Einstein manifolds]{Diameter estimation of $(m,\rho)$-quasi Einstein manifolds}
\author[A. A. Shaikh, P. Mandal, C. K. Mondal]{Absos Ali Shaikh$^{1*}$, Prosenjit Mandal$^2$, Chandan Kumar Mondal$^{3}$}

\address{\noindent\newline $^{1,2}$Department of Mathematics,\newline The University of Burdwan,Golapbag,\newline Purba Bardhaman-713101,\newline West Bengal, India}

\address{\noindent\newline $^3$School of Sciences,\newline Netaji Subhas Open University,\newline Durgapur Regional Center, Durgapur-713214\newline Paschim Bardhaman,\newline West Bengal, India}

\email{$^1$aask2003@yahoo.co.in, aashaikh@math.buruniv.ac.in}
\email{$^2$prosenjitmandal235@gmail.com}
\email{$^3$chan.alge@gmail.com, chandanmondal@wbnsou.ac.in}

\begin{document}
\begin{abstract}
This paper aims to study  the $(m,\rho)$-quasi Einstein manifold. This article shows that a complete and connected Riemannian manifold under certain conditions becomes compact. Also, we have determined an upper bound of the diameter for such a manifold. It is also exhibited that the potential function acquiesces to the Hodge-de Rham potential up to a real constant in an $(m,\rho)$-quasi Einstein manifold. Later, some triviality and integral conditions are established for a non-compact complete $(m,\rho)$-quasi Einstein manifold having finite volume. Finally, it is proved that with some certain constraints, a complete Riemannian manifold admits finite fundamental group. Furthermore, some conditions for compactness criteria have also been deduced.
\end{abstract}
\noindent\footnotetext{$^*$ Corresponding author.\\ $\mathbf{2020}$\hspace{5pt}Mathematics\; Subject\; Classification: 53C20; 53C21; 53C25; 53E20.\\ 
{Key words and phrases: Riemannian manifold; $(m,\rho)$-Quasi Einstein Manifold; scalar curvature; index form; diameter estimation.} }
\maketitle
\section{Introduction and preliminaries}
  A non-gradient generalized $m$-quasi-Einstein manifold (\cite{BG16}) is an $n(> 2)$-dimensional Riemannian manifold $(N,g)$ such that it preserves a smooth function $\alpha \in C^{\infty}(N)$, with

\begin{equation}\label{mr2}
\alpha g+\frac{1}{m}W^\flat\otimes W^\flat=\frac{1}{2}\mathcal{L}_Wg+Ric,
\end{equation}
where $W\in \chi{(N)}$, it's dual 1-form $W^\flat$, $\mathcal{L}_Wg$ indicates the Lie derivative of $g$ along $W$ and $m$ is a scalar with $0<m\leq\infty$. In $(\ref{mr2})$, if we replace $\alpha$ by $\lambda+\rho R,$ for some real constants $\lambda$, $\rho$ and the scalar curvature $R$ of $N$, then $(N,g)$ is called a non-gradient $(m,\rho)$-quasi Einstein manifold (\cite{RP21}), and in this case equation $(\ref{mr2})$ reduces to
\begin{equation}\label{q1}
\lambda g+\rho Rg+\frac{1}{m}W^\flat\otimes W^\flat=\frac{1}{2}\mathcal{L}_Wg+Ric.
\end{equation}
If $W$ is the gradient of a real valued smooth function $f$ on $N$, then the $(m,\rho)$-quasi Einstein manifold is said to be a (gradient) $(m,\rho)$-quasi Einstein manifold (for details see \cite{DG17,HW13} ), with potential function $f$. In this case, ($\ref{q1}$) reduces to
\begin{equation}\label{e0}
\lambda g+\rho Rg+\frac{1}{m}df\otimes df=\nabla^2f+Ric.
\end{equation}
 If $m=\infty$ and $\rho=0$ (resp., $m=\infty$), then (\ref{e0}) compresses to the Ricci soliton equation (see e.g. \cite{CK04,DA20,HA82,CA20,SMM20,SD22}) \big(resp., $\rho$-Einstein soliton equation (see e.g. \cite{AAP2021,SMM2022})\big). In (\ref{e0}), if $m$ and $\lambda$ is taken from $\chi{(N)}$, then $(N,g)$ is known as the generalized $m$-quasi Einstein manifold \cite{CAT12}.
 Furthermore, tracing (\ref{e0}), we obtain
\begin{equation}\label{e1}
\lambda n+n\rho R+\frac{1}{m} |\nabla f|^2=\Delta f+R.
\end{equation}
   For a fixed $\phi \in C^{\infty}(N)$, the weighted Laplacian is defined by (see, \cite{CC2010}) $$\Delta_\phi f=\Delta f-\langle \nabla \phi,\nabla f \rangle.$$

Catino \cite{CAT12} introduced the notion of a generalized quasi-Einstein manifold and characterized such type of manifolds with harmonic Weyl tensor. For a compact $m$-quasi Einstein metric having constant scalar curvature, a triviality result is proved by Case et al. \cite{CSW11}. They also determined that all compact $2$-dimensional $m$-quasi Einstein manifolds are trivial. For a compact gradient $\rho$-Einstein soliton with some conditions, Shaikh et al. \cite{SMM2022} have provided a lower bound of the diameter. 
 Huang and Wei \cite{HW13} have deduced some rigidity results on a compact $(m,\rho)$-quasi Einstein manifolds, in particular, by using some conditions on scalar curvature or on the constant $\rho$, they have showed that such a manifold is trivial. Demirba\'g and G\'uler \cite{DG17} have characterized an $(m, \rho)$-quasi Einstein manifold admitting closed conformal or
parallel vector field. Also, they have established some rigidity results giving some new examples of the $(m, \rho)$-quasi Einstein manifold.

L\'opez and Rio \cite{LR2008} have proved a compactness theorem for a complete Riemannian manifold satisfying $\mathcal{L}_Wg+Ric\geq \lambda g$, with potential vector field $W$ having bounded norm. Wylie \cite{WW2008} has proved that the fundamental group of a complete Riemannian manifold satisfying $\mathcal{L}_Wg+Ric\geq \lambda g$, for $\lambda >0$, is finite.

  Hence inspiring by the above studies and the study of \cite{ABR2011}, in this article, we have showed the following:
  
  \noindent In the first theorem we have generalized the work of Limoncu \cite{Limoncu2012}, and obtained an upper bound of the diameter.
  \begin{theorem}\label{theoremd1}
  Let $(N, g)$ be a complete and connected Riemannian manifold satisfying $Ric+\nabla^2f-\frac{1}{m}df\otimes df\geq\lambda g+\rho Rg$. If $|f|\leq K$ and $\rho R \geq K_1$, for some real constants $K, K_1$ choosing in such a way that $(\lambda+K_1)>0,$ then $N$ is compact and the diameter satisfies 
  \begin{equation}
  diam (N)\leq \pi\sqrt{2}\sqrt{\frac{(n-1)+K\sqrt{2}}{\lambda+K_1}}.
  \end{equation}
  \end{theorem}
  Next, we have generalized the work of  Li and Wenyi \cite{CC2010}, and obtained the following:
  \begin{theorem}\label{theoremd}
  Assume that $(N,g,e^{-f} dvol)$ is a smooth metric measure space satisfying $Ric+\nabla^2f-\frac{1}{m}df\otimes df\geq\lambda g+\rho Rg$, and $u$ is a positive smooth $f$-harmonic function on $N$, then 
  \begin{equation}
  \frac{1}{2}\Delta F\geq (\lambda+\rho R)F+ \frac{1}{n}F^2-\frac{2}{n} |\nabla f| F^{\frac{3}{2}}-|\nabla F| F^{\frac{1}{2}}+\frac{1}{m}\{df (\nabla \omega)\}^2+\frac{1}{2}\langle \nabla f, \nabla F\rangle,
  \end{equation}
  where we denote $F=|\nabla \log u|^2$.
  \end{theorem}
\begin{theorem}\label{th3}
If $(N,g)$ is a compact and oriented $(m,\rho)$-quasi Einstein manifold, then its potential function is given by $f=\sigma+C_1$, where $C_1$ is a constant and $\sigma$ is the Hodge-de Rham potential .
\end{theorem}
Let $(N,g)$ be an oriented Riemannian manifold and $A^k(N)$ be the collection of all $k$-th differential forms in $N$. Now for given integer $k\geq 0$, the global inner product in $A^k(N)$ is defined by (see, \cite{MO01})
$$\langle \zeta,\omega\rangle=\int_N \zeta\wedge *\omega,$$
for $\zeta,\omega\in A^k(N)$. Here $``*" $ denotes the Hodge star operator. The global norm of $\zeta\in A^k(N)$ is given by $\|\zeta\|^2=\langle \zeta,\zeta\rangle$ with $\|\zeta\|^2\leq \infty$.
\begin{theorem}\label{th4}
Let $(N,g)$ be a complete $(m,\rho)$-quasi Einstein manifold which is non-compact and of finite volume. If $\rho>\frac{1}{n}$ and also $W$ admits finite global norm, then the following holds:
\begin{itemize}
\item[(i)] If $R\geq\frac{\lambda n}{1-\rho n}$, then $N$ is trivial,
\item[(ii)]If $(\lambda+\rho)R\geq 0$, then $\frac{1}{m}\int_{N}|W|^2dV\leq\int_{N}R dV,$ and 
\item[(iii)] $\int_{N} RdV \leq \frac{\lambda n}{1-\rho n} Vol(N),$ where $Vol(N)$ represents the volume of $N.$
\end{itemize}
\end{theorem}
By virtue of Theorem \ref{th4}, we deduce the following corollary:
\begin{cor}\label{co1}
Let $(N,g)$ be a complete $(m,\rho)$-quasi Einstein manifold which is non-compact and of finite volume. If $\rho<\frac{1}{n}$ and $W$ is of finite global norm, then the following holds:
\begin{itemize}
\item[(i)] If $R\leq\frac{\lambda n}{1-\rho n}$, then $N$ is trivial,
\item[(ii)]If $(\lambda+\rho)R\leq 0$, then $\int_{N}R dV \leq \frac{1}{m}\int_{N}|X|^2 dV,$ and 
\item[(iii)] $\int_{N} RdV \geq \frac{\lambda n}{1-\rho n} Vol(N).$
\end{itemize}
\end{cor}
Next, we are interested in studying a complete Riemannian manifold $(N,g)$ with a vector field $W$ such that,
\begin{equation}\label{a7}
\lambda g+\rho Rg+\frac{1}{m}W^\flat\otimes W^\flat \leq\frac{1}{2}\mathcal{L}_Wg+Ric,
\end{equation}
In $(\ref{a7})$ the equality gives the fundamental equation of  an $(m,\rho)$-quasi Einstein manifold. Here the main aim is to prove the following results:
\begin{theorem}\label{th1}
If $(N,g)$ is a complete Riemannian manifold satisfying (\ref{a7}) along with $\int_{0}^{r}\rho R=K_3 $, for some constant $K_3$, then the fundamental group of $N$ is finite.
\end{theorem}
\begin{theorem}\label{lemma1}
Let $(N, g)$ be a complete Riemannian manifold satisfying $(\ref{a7})$ with $\|W\|$ bounded and $\int_{0}^{r}\rho R=K_3 $, for some constant $K_3$ with $(\lambda+K_3)>0$. Then $N$ becomes compact.
\end{theorem}
\section{Proof of the results}
\begin{proof}[\textbf{Proof of Theorem \ref{theoremd1}}]
Let $p, q\in N$ and let $\Gamma$ be a minimizing unit speed geodesic segment
from $p$ to $q$ of length $d$. Considering a parallel orthonormal frame $\{\xi_1=\Gamma', \xi_2,\cdots, \xi_n\}$
along $\Gamma$ and a smooth function $h\in C^{\infty}
([0,d])$ such that $h(0)=h(d)=0$, we get (see, \cite{Limoncu2012}),
\begin{equation}\label{d1}
\sum_{i=2}^{n} \eta(h\xi_i,h\xi_i)=\int_{0}^{d}\{(n-1)h'^2-h^2 Ric(\Gamma', \Gamma')\}dt,
\end{equation}
where $\eta$ denotes the index form of $\Gamma$.\\
With our assumption the relation $(\ref{d1})$, yields 
\begin{eqnarray}\label{d2}
\sum_{i=2}^{n} \nonumber\eta(h\xi_i,h\xi_i)&\leq&\int_{0}^{d}\left[(n-1)h'^2+h^2 \{\nabla^2f-\frac{1}{m}df\otimes df-\lambda g-\rho Rg\}(\Gamma', \Gamma')\right]dt\\
&=&\int_{0}^{d}\left[(n-1)h'^2+h^2\nabla^2f(\Gamma', \Gamma')-\frac{h^2}{m}\left(df(\Gamma')\right)^2-h^2(\lambda +\rho R)\right]dt.
\end{eqnarray}
Again, since $|f|\leq K$, we obtain (see, \cite{Limoncu2012}),
\begin{eqnarray}\label{d3}
\int_{0}^{d}h^2\nabla^2f(\Gamma', \nonumber\Gamma')dt&=&\int_{0}^{d}h^2g(\nabla_{\Gamma'}\nabla f, \Gamma')dt\\
\nonumber&=&\int_{0}^{d}h^2 \Gamma'(g(\nabla f, \Gamma'))dt\\
&\leq& 2K\sqrt{d}\left(\int_{0}^{d}\left(\frac{d}{dt}(hh')\right)^2 dt\right)^{\frac{1}{2}}.
\end{eqnarray}
The equations $(\ref{d2})$ and $(\ref{d3})$ together implies
\begin{eqnarray}\label{d4}
\sum_{i=2}^{n} \nonumber\eta(h\xi_i,h\xi_i)\leq&&\int_{0}^{d}(n-1)h'^2 dt+2K\sqrt{d}\left(\int_{0}^{d}\left(\frac{d}{dt}(hh')\right)^2 dt\right)^{\frac{1}{2}}\\&&-\int_{0}^{d}\frac{h^2}{m}\left(df(\Gamma')\right)^2 dt-\int_{0}^{d}h^2(\lambda +\rho R)dt.
\end{eqnarray}
If we take $h(t)=sin(\frac{\pi}{d}t)$, then $(\ref{d4})$, yields
\begin{eqnarray}\label{d5}
\sum_{i=2}^{n} \nonumber\eta(h\xi_i,h\xi_i)\leq&&(n-1)\int_{0}^{d}\frac{\pi^2}{d^2}cos^2\left(\frac{\pi}{d}t\right) dt+2K\sqrt{d}\left(\int_{0}^{d}\left(\frac{\pi^2}{d^2}cos\left(\frac{2\pi}{d}t\right)\right)^2
 dt\right)^{\frac{1}{2}}\\\nonumber&&-\lambda\int_{0}^{d}sin^2 \left(\frac{\pi}{d}t\right) dt -\int_{0}^{d}sin^2\left(\frac{\pi}{d}t\right)\rho Rdt\\
\leq\nonumber&&(n-1)\frac{\pi^2}{2d}+\frac{K\pi^2\sqrt{2}}{d}-\frac{\lambda}{2}d -\frac{K_1}{2}d.
\end{eqnarray}
Since $\Gamma$ is a minimizing geodesic, hence $\eta(h\xi_i,h\xi_i)\geq 0$. Consequently, 
\begin{equation*}
0\leq(n-1)\frac{\pi^2}{2d}+\frac{K\pi^2\sqrt{2}}{d}-\frac{\lambda}{2}d -\frac{K_1}{2}d.
\end{equation*}
The above relation entails
\begin{equation*}
d\leq \pi\sqrt{2}\sqrt{\frac{(n-1)+K\sqrt{2}}{\lambda+K_1}}.
\end{equation*}
This completes the proof.
\end{proof}
\begin{proof}[\textbf{Proof of Theorem \ref{theoremd}}]
  Set $\omega=\log u$, then $F=|\nabla \omega|^2$ and $\omega$ satisfies the equation:
  \begin{equation*}
  \Delta_f \omega +|\nabla \omega|^2=0.
  \end{equation*}
  Thus, 
  \begin{equation*}
  \Delta_f \omega=-F.
  \end{equation*}
  Therefore,
  \begin{equation}\label{b1}
  |\nabla^2 \omega|^2\geq \frac{1}{n}|\Delta \omega|^2 =\frac{1}{n}|F-\langle \nabla f, \nabla \omega \rangle|^2\geq \frac{1}{n}F^2-\frac{2}{n}\langle \nabla f, \nabla \omega \rangle F,
  \end{equation}
  and
  \begin{equation}\label{b2}
  \langle \nabla \omega, \nabla \Delta_f \omega \rangle=-\langle \nabla \omega, \nabla F \rangle.
  \end{equation}
  Now, the weighted Bochner formula:
  $$\frac{1}{2}\Delta_f |\nabla \omega|^2=|\nabla^2 \omega|^2+\langle \nabla \omega, \nabla \Delta_f \omega \rangle+Ric_f(\nabla \omega,\nabla \omega),$$ together with $(\ref{b1})$, $(\ref{b2})$ and our assumption, yields
  $$\frac{1}{2}\Delta_f F\geq \frac{1}{n}F^2-\frac{2}{n}\langle \nabla f, \nabla \omega \rangle F-\langle \nabla \omega, \nabla F \rangle +(\lambda+\rho R)|\nabla \omega|^2+\frac{1}{m}df\otimes df (\nabla \omega,\nabla \omega).$$
  This implies
  \begin{equation}\label{b3}
  \frac{1}{2}\Delta F-\langle \nabla f, \nabla F\rangle\geq \frac{1}{n}F^2-\frac{2}{n}\langle \nabla f, \nabla \omega \rangle F-\langle \nabla \omega, \nabla F \rangle +(\lambda+\rho R)F+\frac{1}{m}\{df (\nabla \omega)\}^2.
  \end{equation}
  Applying the Cauchy-Schwarz inequality, we obtain
  \begin{equation*}
  \langle \nabla \omega, \nabla f \rangle \leq |\nabla \omega||\nabla f|=|\nabla f| F^{\frac{1}{2}}.
  \end{equation*}
  Using the last inequality in $(\ref{b3})$, we get the estimation.
  \end{proof}
\begin{proof}[\textbf{Proof of Theorem \ref{th3}}]
If $W\in \chi (N)$, then as a consequence of Hodge-de Rham decomposition theorem, (see e.g. \cite{WF1983}), $W$ can be written as
\begin{equation}\label{hd1}
W-\nabla \sigma=Y,
\end{equation}
where $\sigma \in \chi (N)$ is the Hodge-de Rham potential and  $div\ Y=0$. By considering an $(m,\rho)$-quasi Einstein manifold $(N,g)$ so that the equation (\ref{q1}) holds, then we obtain
\begin{equation}\label{qe1}
R+div W=\lambda n+\frac{1}{m}|W|^2+\rho Rn.
\end{equation}
Therefore ($\ref{hd1}$) concludes that $div W=\Delta \sigma$ and hence $(\ref{qe1})$ entails
\begin{equation}\label{qe2}
\lambda n+\frac{1}{m}|W|^2+\rho Rn=R+\Delta \sigma.
\end{equation}
Again $(\ref{e0})$ indicates that
 \begin{equation}\label{qe3}
 \lambda n+\frac{1}{m}|\nabla f|^2+\rho Rn=R+\Delta f.
 \end{equation}
 From $(\ref{qe2})$ and $(\ref{qe3})$, we have 
 \begin{equation*}
 \Delta(f-\sigma)=0,
 \end{equation*}
 which follows the result.
\end{proof}
\begin{proof}[\textbf{Proof of Theorem \ref{th4}}]
We have, for $l>0$ 
\begin{eqnarray}
\nonumber\frac{1}{l}\int_{\mathscr{B}(p,2l)}|W|dV & \leq &  \Big(\int_{\mathscr{B}(p,2l)}\Big(\frac{1}{l} \Big)^2 dV\Big)^{1/2}\Big(\int_{\mathscr{B}(p,2l)}\langle W,W\rangle dV \Big)^{1/2}\\
\nonumber&\leq& \frac{1}{l}\Big(Vol(N)\Big)^{1/2}\|W\|_{\mathscr{B}(p,2l)},
\end{eqnarray}
where $\mathscr{B}(p,l)$ is the open ball of radius $l$ and center at $p$. Thus
\begin{equation*}
\liminf_{\ l\rightarrow \infty} \frac{1}{l}\int_{\mathscr{B}(p,2l)}|W|dV=0.
\end{equation*}
Again, a Lipschitz continuous function $\omega_l$ exists, (see \cite{YA76}) such that for some real constant $K_2>0$,
\begin{eqnarray*}
&& 0\leq \omega_l(x)\leq 1\quad\forall x\in N,\\
&& \text{ supp }\omega_l\subset \mathscr{B}(p,2l),\\
&&|d\omega_l|\leq \frac{K_2}{l}\qquad \text{  almost everywhere on }N \text{ and}\\
&&\omega_l(x)=1\quad\forall x\in \mathscr{B}(p,l).
\end{eqnarray*}
As $\lim\limits_{l\rightarrow\infty}\omega_l=1$, using the function $\omega_l$, we get
\begin{equation*}
\frac{C}{l}\int_{\mathscr{B}(p,2l)}|W|dV\geq\Big|\int_{\mathscr{B}(p,2l)}\omega_l div W dV \Big|.
\end{equation*}
From the defining relation of $(m,\rho)$-quasi Einstein manifold, we obtain
\begin{equation}\label{mr1}
\int_N \{\lambda n+\frac{1}{m} |W|^2+(\rho n-1)R\}dV=0.
\end{equation}
By  virtue of equation (\ref{mr1}) and our assumption, we get the desired results.
\end{proof}
To prove Theorem \ref{th1}, we need the following results:
\begin{lemma}[\cite{WW2008}]\label{a4}
Let $(N,g)$ be a complete Riemannian manifold and $p,q \in N$ with $r=d(p,q)>1$. If $\Gamma$ is a minimal geodesic joining from $p$ to $q$ and parametrized by the arc length s, then
\begin{equation}
2(n-1)+\mathscr{E}_p+\mathscr{E}_q\geq\int_{0}^{r}Ric(\Gamma'(s),\Gamma'(s))ds,  
\end{equation}
where
$\mathscr{E}_q=max\{0, sup\{Ric_y(v,v): y \in B(q,1), ||v||=1\}\},$ for q $\in$ $N$.
\end{lemma} 
\begin{lemma}\label{l1}
If $(N,g)$ is a complete Riemannian manifold satisfying (\ref{a7}), then for any p,q $\in N$,
\begin{equation}\label{a5}
 max\{1,\frac{1}{K_3+\lambda}\big(2(n-1)+\mathscr{E}_p+\mathscr{E}_q+||W_p||+||W_q|| \big)\}\geq d(p,q).
\end{equation} 
\end{lemma}
\begin{proof}
Suppose that $d(p,q)>1$ and $\Gamma$ is a minimal geodesic joining from p to q and parametrized by arc length s. Then along $\Gamma$, we can easily get
\begin{eqnarray*}
\mathcal{L}_W g(\Gamma', \Gamma') = 2 \frac{d}{ds}[g(W,\Gamma')].
\end{eqnarray*}
Applying Lemma \ref{a4}, we obtain
\begin{equation}\label{a3}
2(n-1)+\mathscr{E}_p+\mathscr{E}_q \geq\int_{0}^{r}Ric(\Gamma'(s),\Gamma'(s))ds.  
\end{equation}
 Using the Cauchy-Schwarz inequality and (\ref{a7}), we have
\begin{eqnarray*}
\int_{0}^{r}Ric(\Gamma'(s), \Gamma'(s))ds &\geq & g_p(W,\Gamma'(0))-g_q(W, \Gamma'(r))+\int_{0}^{r}(W^{\flat}(\Gamma'(s)))^2 ds+\lambda d(p,q)+ \int_{0}^{r} \rho R  ds\\
&\geq& \lambda d(p,q)-||W_p||-||W_q|| + K_3 d(p,q).
\end{eqnarray*}
In view of (\ref{a3}) and solving for $d(p,q)$, the last inequality  entails (\ref{a5}).
\end{proof}
\begin{proof}[\textbf{Proof of Theorem \ref{th1}}]
We omit the proof of Theorem \ref{th1}, as it is straightforward from the Lemma \ref{l1} and following the technique of \cite[Theorem 1.1]{WW2008}.
\end{proof}
\begin{proof}[\textbf{Proof of Theorem \ref{lemma1}}]
Let $p \in N$ and $\Gamma : [0,\infty ) \rightarrow N $ be any geodesic starting from $p$ and parametrized by the arc length $s$.
Hence (\ref{a7}) and the Cauchy-Schwarz inequality together imply that
\begin{eqnarray*}
\int_{0}^{r} Ric(\Gamma'(s), \Gamma'(s)) ds &\geq& \int_{0}^{r}(W^{\flat}(\Gamma'(s)))^2+\lambda r + \int_{0}^{r} \rho R g + 2g(W_p,\Gamma'(0))-2g(W_{\Gamma(r)}, \Gamma'(r)) ds\\
&\geq& 2g(W_p, \Gamma'(0))-2\|W_{\Gamma(r)}\|+\lambda r + K_3r.
\end{eqnarray*}
The boundedness of $\|W\|$ provides
\begin{equation*}
\int_{0}^{+\infty} Ric(\Gamma'(s),\Gamma'(s))= + \infty,
\end{equation*} 
and hence Ambrose's compactness criteria \cite {WA1957} concludes that $N$ is compact.
\end{proof}
\section*{Acknowledgment}
 The second author gratefully acknowledges to the CSIR(File No.:09/025(0282)/2019-EMR-I), Govt. of India for the award of JRF. Also the third author conveys sincere thanks to the
 Netaji Subhas Open University for partial financial assistance (Project No.: AC/140/2021-22).

\end{document}